\documentclass{amsart}
\usepackage{amsmath}
\usepackage{amsfonts}

\newtheorem{theorem}{Theorem}
\theoremstyle{plain}

\newtheorem{corollary}{Corollary}

\newtheorem{lemma}{Lemma}

\newtheorem{remark}{Remark}

\numberwithin{equation}{section}
\input{tcilatex}

\begin{document}
\title[Integral inequalities]{A new generalization of some integral
inequalities for $(\alpha ,m)-$convex functions}
\author{\.{I}mdat \.{I}\c{s}can}
\address{Department of Mathematics, Faculty of Arts and Sciences,\\
Giresun University, 28100, Giresun, Turkey.}
\email{imdat.iscan@giresun.edu.tr}
\date{July 01, 2012}

\begin{abstract}
In this paper, we derive new estimates for the remainder term of the
midpoint, trapezoid, and Simpson formulae for functions whose derivatives in
absolute value at certain power are $(\alpha ,m)-$convex.
\end{abstract}

\maketitle

\section{Introduction}

Let $f:I\subseteq \mathbb{R\rightarrow R}$ be a convex function defined on
the interval $I$ of real numbers and $a,b\in I$ with $a<b$. The following
double inequality is well known in the literature as Hermite-Hadamard
integral inequality 
\begin{equation}
f\left( \frac{a+b}{2}\right) \leq \frac{1}{b-a}\dint\limits_{a}^{b}f(x)dx%
\leq \frac{f(a)+f(b)}{2}\text{.}  \label{1-1}
\end{equation}

The class of $(\alpha ,m)-$convex functions was first introduced In \cite%
{M93}, and it is defined as follows:

The function $f:\left[ 0,b\right] \mathbb{\rightarrow R}$, $b>0$, is said to
be $(\alpha ,m)$-convex where $(\alpha ,m)\in \left[ 0,1\right] ^{2}$, if we
have$f$%
\begin{equation*}
\left( tx+m(1-t)y\right) \leq t^{\alpha }f(x)+m(1-t^{\alpha })f(y)
\end{equation*}

for all $x,y\in \left[ 0,b\right] $ and $t\in \left[ 0,1\right] $.

It can be easily that for $(\alpha ,m)\in \left\{ (0,0),(\alpha
,0),(1,0),(1,m),(1,1),(\alpha ,1)\right\} $ one obtains the following
classes of functions: increasing, $\alpha $-starshaped, starshaped, $m$%
-convex, convex, $\alpha $-convex.

Denote by $K_{m}^{\alpha }(b)$ the set of all $(\alpha ,m)$-convex functions
on $\left[ 0,b\right] $ for which $f(0)\leq 0$. For recent results and
generalizations concerning $(\alpha ,m)$-convex functions (see \cite%
{BOP08,M93,OAK11,OKS10,OSS11,SSOR09}).

In \cite{SSOR09}, Set et al. proved the following Hadamard type inequality
for $(\alpha ,m)$-convex functions:

\begin{theorem}
\label{A.1}Let $f:\left[ 0,\infty \right) \rightarrow \mathbb{R}$ be an $%
(\alpha ,m)$-convex function with $(\alpha ,m)\in \left( 0,1\right] ^{2}.$
If $0\leq a<b<\infty $ and $f\in L[a,b],$ then one has the inequality:%
\begin{equation}
\frac{1}{b-a}\dint\limits_{a}^{b}f(x)dx\leq \min \left\{ \frac{f(a)+\alpha
mf\left( \frac{b}{m}\right) }{\alpha +1},\frac{f(b)+\alpha mf\left( \frac{a}{%
m}\right) }{\alpha +1}\right\} .  \label{1-3}
\end{equation}
\end{theorem}

The following inequality is well known in the literature as Simpson's
inequality .

Let $f:\left[ a,b\right] \mathbb{\rightarrow R}$ be a four times
continuously differentiable mapping on $\left( a,b\right) $ and $\left\Vert
f^{(4)}\right\Vert _{\infty }=\underset{x\in \left( a,b\right) }{\sup }%
\left\vert f^{(4)}(x)\right\vert <\infty .$ Then the following inequality
holds:%
\begin{equation*}
\left\vert \frac{1}{3}\left[ \frac{f(a)+f(b)}{2}+2f\left( \frac{a+b}{2}%
\right) \right] -\frac{1}{b-a}\dint\limits_{a}^{b}f(x)dx\right\vert \leq 
\frac{1}{2880}\left\Vert f^{(4)}\right\Vert _{\infty }\left( b-a\right) ^{2}.
\end{equation*}%
\ \qquad In recent years many authors have studied error estimations for
Simpson's inequality; for refinements, counterparts, generalizations and new
Simpson's type inequalities, see \cite{P12,SA11,SSO10,SSO10a}.

In this paper, in order to provide a unified approach to establish midpoint
inequality, trapezoid inequality and Simpson's inequality for functions
whose derivatives in absolute value at certain power are $(\alpha ,m)$%
-convex, we derive a general integral identity for convex functions.

\section{Main results}

In order to generalize the classical Trapezoid, midpoint and Simpson type
inequalities and prove them, we need the following Lemma:

\begin{lemma}
\label{2.1}Let $f:I\subseteq \mathbb{R\rightarrow R}$ be a differentiable
mapping on $I^{\circ }$ such that $f^{\prime }\in L[a,b]$, where $a,b\in I$
with $a<b$, $\lambda ,\mu \in \left[ 0,1\right] $ and $m\in \left( 0,1\right]
.$ Then the following equality holds:%
\begin{eqnarray}
&&\lambda \left( \mu f(a)+\left( 1-\mu \right) f(mb)\right) +\left(
1-\lambda \right) f(\mu a+m\left( 1-\mu \right) b)-\frac{1}{mb-a}%
\dint\limits_{a}^{mb}f(x)dx  \label{2-1} \\
&=&\left( mb-a\right) \left[ \dint\limits_{0}^{\mu }\left[ -t+\lambda \left(
1-\mu \right) \right] f^{\prime }\left( ta+m(1-t)b\right) dt\right.  \notag
\\
&&\left. +\dint\limits_{\mu }^{1}\left[ -t+\left( 1-\alpha \lambda \right) %
\right] f^{\prime }\left( ta+m(1-t)b\right) dt\right] .  \notag
\end{eqnarray}
\end{lemma}

A simple proof of the equality can be done by performing an integration by
parts in the integrals from the right side and changing the variable. The
details are left to the interested reader.

\begin{theorem}
\label{2.2}Let $f:I\subseteq \lbrack 0,\infty )\rightarrow \mathbb{R}$ be a
differentiable mapping on $I^{\circ }$ such that $f^{\prime }\in L[a,b]$,
where $a,b\in I^{\circ }$ with $a<b$ and $\lambda ,\mu \in \left[ 0,1\right] 
$. If $\left\vert f^{\prime }\right\vert ^{q}$ is $(\alpha ,m)-$convex on $%
[a,b]$, for $(\alpha ,m)\in \left( 0,1\right] ^{2}$, $mb>a$, $q\geq 1,$ then
the following inequality holds:%
\begin{eqnarray}
&&\left\vert \lambda \left( \mu f(a)+\left( 1-\mu \right) f(mb)\right)
+\left( 1-\lambda \right) f(\mu a+m\left( 1-\mu \right) b)-\frac{1}{mb-a}%
\dint\limits_{a}^{mb}f(x)dx\right\vert  \label{2-2} \\
&\leq &\left\{ 
\begin{array}{cc}
\begin{array}{c}
\left( mb-a\right) \left\{ \varepsilon _{2}^{1-\frac{1}{q}}\left( \delta
_{3}\left\vert f^{\prime }(a)\right\vert ^{q}+m\delta _{4}\left\vert
f^{\prime }(b)\right\vert ^{q}\right) ^{\frac{1}{q}}\right. \\ 
\left. +\varepsilon _{3}^{1-\frac{1}{q}}\left( \beta _{1}\left\vert
f^{\prime }(a)\right\vert ^{q}+m\beta _{2}\left\vert f^{\prime
}(b)\right\vert ^{q}\right) ^{\frac{1}{q}}\right\}%
\end{array}%
, & \lambda \left( 1-\mu \right) \leq \mu \leq 1-\lambda \mu \\ 
\begin{array}{c}
\left( mb-a\right) \left\{ \varepsilon _{1}^{1-\frac{1}{q}}\left( \delta
_{1}\left\vert f^{\prime }(a)\right\vert ^{q}+m\delta _{2}\left\vert
f^{\prime }(b)\right\vert ^{q}\right) ^{\frac{1}{q}}\right. \\ 
\left. +\varepsilon _{3}^{1-\frac{1}{q}}\left( \beta _{1}\left\vert
f^{\prime }(a)\right\vert ^{q}+m\beta _{2}\left\vert f^{\prime
}(b)\right\vert ^{q}\right) ^{\frac{1}{q}}\right\}%
\end{array}%
, & \mu \leq \lambda \left( 1-\mu \right) \leq 1-\lambda \mu \\ 
\begin{array}{c}
\left( mb-a\right) \left\{ \varepsilon _{2}^{1-\frac{1}{q}}\left( \delta
_{3}\left\vert f^{\prime }(a)\right\vert ^{q}+m\delta _{4}\left\vert
f^{\prime }(b)\right\vert ^{q}\right) ^{\frac{1}{q}}\right. \\ 
\left. +\varepsilon _{4}^{1-\frac{1}{q}}\left( \beta _{3}\left\vert
f^{\prime }(a)\right\vert ^{q}+m\beta _{4}\left\vert f^{\prime
}(b)\right\vert ^{q}\right) ^{\frac{1}{q}}\right\}%
\end{array}%
, & \lambda \left( 1-\mu \right) \leq 1-\lambda \mu \leq \mu%
\end{array}%
\right. ,  \notag
\end{eqnarray}%
where%
\begin{equation*}
\varepsilon _{1}=-\frac{\mu ^{2}}{2}+\lambda (1-\mu )\mu ,\ \varepsilon _{2}=%
\left[ \lambda (1-\mu )\right] ^{2}-\varepsilon _{1},
\end{equation*}%
\begin{eqnarray*}
\varepsilon _{3} &=&(1-\lambda \mu )^{2}-(1-\lambda \mu )\left( 1+\mu
\right) +\frac{1+\mu ^{2}}{2}, \\
\varepsilon _{4} &=&\frac{1-\mu ^{2}}{2}-(1-\lambda \mu )\left( 1-\mu
\right) ,
\end{eqnarray*}%
\begin{eqnarray*}
\delta _{1} &=&\frac{\lambda \left( 1-\mu \right) \mu ^{\alpha +1}}{\alpha +1%
}-\frac{\mu ^{\alpha +2}}{\alpha +2}, \\
\delta _{2} &=&\lambda (1-\mu )\mu -\frac{\mu ^{2}}{2}-\delta _{1}, \\
\delta _{3} &=&\frac{2\left[ \lambda \left( 1-\mu \right) \right] ^{\alpha
+2}}{\left( \alpha +1\right) \left( \alpha +2\right) }-\frac{\lambda \left(
1-\mu \right) \mu ^{\alpha +1}}{\alpha +1}+\frac{\mu ^{\alpha +2}}{\alpha +2}%
, \\
\delta _{4} &=&\left[ \lambda \left( 1-\mu \right) \right] ^{2}-\lambda \mu
\left( 1-\mu \right) +\frac{\mu ^{2}}{2}-\delta _{3},
\end{eqnarray*}%
\begin{eqnarray*}
\beta _{1} &=&\frac{2\left( 1-\lambda \mu \right) ^{\alpha +2}}{\left(
\alpha +1\right) \left( \alpha +2\right) }-\frac{\left( 1-\lambda \mu
\right) \left( 1+\mu ^{\alpha +1}\right) }{\alpha +1}+\frac{1+\mu ^{\alpha
+2}}{\alpha +2}, \\
\beta _{2} &=&\left( 1-\lambda \mu \right) ^{2}-\left( 1-\lambda \mu \right)
\left( 1+\mu \right) +\frac{1+\mu ^{2}}{2}-\beta _{1}, \\
\beta _{3} &=&\frac{1-\mu ^{\alpha +2}}{\alpha +2}-\left( 1-\lambda \mu
\right) \frac{1-\mu ^{\alpha +1}}{\alpha +1}, \\
\beta _{4} &=&\left( 1-\lambda \mu \right) \left( \mu -1\right) +\frac{1-\mu
^{2}}{2}-\beta _{3}.
\end{eqnarray*}
\end{theorem}

\begin{proof}
From Lemma \ref{2.1} and using the properties of modulus and the well known
power mean inequality, we have%
\begin{eqnarray*}
&&\left\vert \lambda \left( \mu f(a)+\left( 1-\mu \right) f(mb)\right)
+\left( 1-\lambda \right) f(\mu a+m\left( 1-\mu \right) b)-\frac{1}{mb-a}%
\dint\limits_{a}^{mb}f(x)dx\right\vert \\
&\leq &\left( mb-a\right) \left[ \dint\limits_{0}^{\mu }\left\vert
-t+\lambda \left( 1-\mu \right) \right\vert \left\vert f^{\prime }\left(
ta+m(1-t)b\right) \right\vert dt\right. \\
&&\left. +\dint\limits_{\mu }^{1}\left\vert -t+\left( 1-\alpha \lambda
\right) \right\vert \left\vert f^{\prime }\left( ta+m(1-t)b\right)
\right\vert dt\right] \\
&\leq &\left( mb-a\right) \left\{ \left( \dint\limits_{0}^{\mu }\left\vert
-t+\lambda \left( 1-\mu \right) \right\vert dt\right) ^{1-\frac{1}{q}}\left(
\dint\limits_{0}^{\mu }\left\vert -t+\lambda \left( 1-\mu \right)
\right\vert \left\vert f^{\prime }\left( ta+m(1-t)b\right) \right\vert
^{q}dt\right) ^{\frac{1}{q}}\right.
\end{eqnarray*}%
\begin{equation}
\left. +\left( \dint\limits_{\mu }^{1}\left\vert -t+\left( 1-\lambda \mu
\right) \right\vert dt\right) ^{1-\frac{1}{q}}\left( \dint\limits_{\mu
}^{1}\left\vert -t+\left( 1-\lambda \mu \right) \right\vert \left\vert
f^{\prime }\left( ta+m(1-t)b\right) \right\vert ^{q}dt\right) ^{\frac{1}{q}%
}\right\}  \label{2-2a}
\end{equation}%
Since $\left\vert f^{\prime }\right\vert ^{q}$ is $(\alpha ,m)-$convex on $%
[a,b],$ we know that for $t\in \left[ 0,1\right] $%
\begin{equation*}
\left\vert f^{\prime }\left( ta+m(1-t)b\right) \right\vert ^{q}\leq
t^{\alpha }\left\vert f^{\prime }(a)\right\vert ^{q}+m(1-t^{\alpha
})\left\vert f^{\prime }(b)\right\vert ^{q}
\end{equation*}%
hence, by simple computation%
\begin{equation}
\dint\limits_{0}^{\mu }\left\vert -t+\lambda \left( 1-\mu \right)
\right\vert dt=\left\{ 
\begin{array}{cc}
\varepsilon _{1}, & \mu \leq \lambda (1-\mu ) \\ 
\ \varepsilon _{2}, & \mu \geq \lambda (1-\mu )%
\end{array}%
\right. ,  \label{2-2b}
\end{equation}%
\begin{equation}
\dint\limits_{\mu }^{1}\left\vert -t+\left( 1-\lambda \mu \right)
\right\vert dt=\left\{ 
\begin{array}{cc}
\varepsilon _{3}, & \mu \leq 1-\lambda \mu \\ 
\varepsilon _{4}, & \mu \geq 1-\lambda \mu%
\end{array}%
\right. ,  \label{2-2c}
\end{equation}%
\begin{eqnarray}
&&\dint\limits_{0}^{\mu }\left\vert -t+\lambda \left( 1-\mu \right)
\right\vert \left\vert f^{\prime }\left( ta+m(1-t)b\right) \right\vert ^{q}dt
\notag \\
&\leq &\dint\limits_{0}^{\mu }\left\vert -t+\lambda \left( 1-\mu \right)
\right\vert \left[ t^{\alpha }\left\vert f^{\prime }(a)\right\vert
^{q}+m(1-t^{\alpha })\left\vert f^{\prime }(b)\right\vert ^{q}\right] dt 
\notag \\
&=&\left\{ 
\begin{array}{cc}
\delta _{1}\left\vert f^{\prime }(a)\right\vert ^{q}+m\delta _{2}\left\vert
f^{\prime }(b)\right\vert ^{q}, & \mu \leq \lambda (1-\mu ) \\ 
\delta _{3}\left\vert f^{\prime }(a)\right\vert ^{q}+m\delta _{4}\left\vert
f^{\prime }(b)\right\vert ^{q}, & \mu \geq \lambda (1-\mu )%
\end{array}%
\right. ,  \label{2-2d}
\end{eqnarray}%
and%
\begin{eqnarray}
&&\dint\limits_{\mu }^{1}\left\vert -t+\left( 1-\lambda \mu \right)
\right\vert \left\vert f^{\prime }\left( ta+m(1-t)b\right) \right\vert ^{q}dt
\notag \\
&\leq &\dint\limits_{\mu }^{1}\left\vert -t+\left( 1-\lambda \mu \right)
\right\vert \left[ t^{\alpha }\left\vert f^{\prime }(a)\right\vert
^{q}+m(1-t^{\alpha })\left\vert f^{\prime }(b)\right\vert ^{q}\right] dt 
\notag \\
&=&\left\{ 
\begin{array}{cc}
\beta _{1}\left\vert f^{\prime }(a)\right\vert ^{q}+m\beta _{2}\left\vert
f^{\prime }(b)\right\vert ^{q}, & \mu \leq 1-\lambda \mu \\ 
\beta _{3}\left\vert f^{\prime }(a)\right\vert ^{q}+m\beta _{4}\left\vert
f^{\prime }(b)\right\vert ^{q}, & \mu \geq 1-\lambda \mu%
\end{array}%
\right. .  \label{2-2e}
\end{eqnarray}%
Thus, using (\ref{2-2b})-(\ref{2-2e}) in (\ref{2-2a}), we obtain the
inequality (\ref{2-2}). This completes the proof.
\end{proof}

\begin{corollary}
\label{2.3}Under the assumptions of Theorem \ref{2.2} with $q=1$%
\begin{eqnarray*}
&&\left\vert \lambda \left( \mu f(a)+\left( 1-\mu \right) f(mb)\right)
+\left( 1-\lambda \right) f(\mu a+m\left( 1-\mu \right) b)-\frac{1}{mb-a}%
\dint\limits_{a}^{mb}f(x)dx\right\vert \\
&\leq &\left\{ 
\begin{array}{cc}
\left( mb-a\right) \left\{ \left( \delta _{3}+\beta _{1}\right) \left\vert
f^{\prime }(a)\right\vert +m\left( \delta _{4}+\beta _{2}\right) \left\vert
f^{\prime }(b)\right\vert \right\} , & \lambda \left( 1-\mu \right) \leq \mu
\leq 1-\lambda \mu \\ 
\left( mb-a\right) \left\{ \left( \delta _{1}+\beta _{1}\right) \left\vert
f^{\prime }(a)\right\vert +m\left( \delta _{2}+\beta _{2}\right) \left\vert
f^{\prime }(b)\right\vert \right\} , & \mu \leq \lambda \left( 1-\mu \right)
\leq 1-\lambda \mu \\ 
\left( mb-a\right) \left\{ \left( \delta _{3}+\beta _{3}\right) \left\vert
f^{\prime }(a)\right\vert +m\left( \delta _{4}+\beta _{4}\right) \left\vert
f^{\prime }(b)\right\vert \right\} , & \lambda \left( 1-\mu \right) \leq
1-\lambda \mu \leq \mu%
\end{array}%
\right. .
\end{eqnarray*}
\end{corollary}

\begin{remark}
In Corollary \ref{2.3}, (i) If we choose $\mu =\frac{1}{2}$, $\lambda =\frac{%
1}{3}$ and $\alpha =1,$ we have%
\begin{eqnarray*}
&&\left\vert \frac{1}{6}\left[ f(a)+4f\left( \frac{a+mb}{2}\right) +f(mb)%
\right] -\frac{1}{mb-a}\dint\limits_{a}^{mb}f(x)dx\right\vert \\
&\leq &\frac{5}{72}\left( mb-a\right) \left\{ \left\vert f^{\prime
}(a)\right\vert +m\left\vert f^{\prime }(b)\right\vert \right\}
\end{eqnarray*}%
which is the same of the Simpson type inequality in \cite[Corollary 2.3 (ii)]%
{P12}.

(ii) If we choose $\mu =\frac{1}{2}$, $\lambda =1$ and $\alpha =1,$ we have%
\begin{eqnarray*}
&&\left\vert \frac{f(a)+f(mb)}{2}-\frac{1}{mb-a}\dint\limits_{a}^{mb}f(x)dx%
\right\vert \\
&\leq &\frac{\left( mb-a\right) }{8}\left\{ \left\vert f^{\prime
}(a)\right\vert +m\left\vert f^{\prime }(b)\right\vert \right\}
\end{eqnarray*}%
which is the same of the trapezoid type inequality in \cite[Corollary 2.3 (i)%
]{P12}.

(iii) If we choose $\mu =\frac{1}{2}$, $\lambda =0$ and $\alpha =1,$ we have%
\begin{eqnarray*}
&&\left\vert f\left( \frac{a+mb}{2}\right) -\frac{1}{mb-a}%
\dint\limits_{a}^{mb}f(x)dx\right\vert \\
&\leq &\frac{\left( mb-a\right) }{8}\left\{ \left\vert f^{\prime
}(a)\right\vert +m\left\vert f^{\prime }(b)\right\vert \right\} .
\end{eqnarray*}
\end{remark}

\begin{corollary}
\label{2.3a}Under the assumptions of Theorem \ref{2.2} with $\mu =\frac{1}{2}
$ and $\lambda =\frac{1}{3}$, we have%
\begin{eqnarray*}
&&\left\vert \frac{1}{6}\left[ f(a)+4f\left( \frac{a+mb}{2}\right) +f(mb)%
\right] -\frac{1}{mb-a}\dint\limits_{a}^{mb}f(x)dx\right\vert \\
&\leq &\left( mb-a\right) \left( \frac{5}{72}\right) ^{1-\frac{1}{q}}\left\{
\left( \delta _{3}^{\ast }\left\vert f^{\prime }(a)\right\vert ^{q}+m\left( 
\frac{5}{72}-\delta _{3}^{\ast }\right) \left\vert f^{\prime }(b)\right\vert
^{q}\right) ^{\frac{1}{q}}\right. \\
&&\left. +\left( \beta _{1}^{\ast }\left\vert f^{\prime }(a)\right\vert
^{q}+m\left( \frac{5}{72}-\beta _{1}^{\ast }\right) \left\vert f^{\prime
}(b)\right\vert ^{q}\right) ^{\frac{1}{q}}\right\} ,
\end{eqnarray*}%
where%
\begin{equation*}
\delta _{3}^{\ast }=\frac{2+3^{\alpha +1}\left( 2\alpha +1\right) }{%
6^{\alpha +2}\left( \alpha +1\right) \left( \alpha +2\right) }
\end{equation*}%
and%
\begin{equation*}
\beta _{1}^{\ast }=\frac{2\times 5^{\alpha +2}+6^{\alpha +1}\left( \alpha
-4\right) -3^{\alpha +1}\left( 2\alpha +7\right) }{6^{\alpha +2}\left(
\alpha +1\right) \left( \alpha +2\right) }.
\end{equation*}
\end{corollary}

\begin{remark}
In Corollary \ref{2.3a}, if we take $\alpha =m=1,$ we obtain the following
inequality%
\begin{eqnarray*}
&&\left\vert \frac{1}{6}\left[ f(a)+4f\left( \frac{a+b}{2}\right) +f(b)%
\right] -\frac{1}{b-a}\dint\limits_{a}^{b}f(x)dx\right\vert \\
&\leq &\left( b-a\right) \left( \frac{5}{72}\right) ^{1-\frac{1}{q}}\left\{
\left( \frac{29}{1296}\left\vert f^{\prime }(b)\right\vert ^{q}+\frac{61}{%
1296}\left\vert f^{\prime }(a)\right\vert ^{q}\right) ^{\frac{1}{q}}\right.
\\
&&\left. +\left( \frac{61}{1296}\left\vert f^{\prime }(b)\right\vert ^{q}+%
\frac{29}{1296}\left\vert f^{\prime }(a)\right\vert ^{q}\right) ^{\frac{1}{q}%
}\right\} ,
\end{eqnarray*}%
which is the same of the inequality in \cite[Theorem 10]{SSO10} for $s=1$ .
\end{remark}

\begin{corollary}
Under the assumptions of Theorem \ref{2.2} with $\mu =\frac{1}{2}$ and $%
\lambda =1$, we have%
\begin{equation*}
\left\vert \frac{f(a)+f(mb)}{2}-\frac{1}{mb-a}\dint\limits_{a}^{mb}f(x)dx%
\right\vert \leq \left( mb-a\right) \left( \frac{1}{8}\right) ^{1-\frac{1}{q}%
}
\end{equation*}%
\begin{eqnarray*}
&&\times \left\{ \left( \frac{1}{2^{\alpha +2}\left( \alpha +1\right) \left(
\alpha +2\right) }\left\vert f^{\prime }(a)\right\vert ^{q}+m\left( \frac{1}{%
8}-\frac{1}{2^{\alpha +2}\left( \alpha +1\right) \left( \alpha +2\right) }%
\right) \left\vert f^{\prime }(b)\right\vert ^{q}\right) ^{\frac{1}{q}%
}\right. \\
&&\left. +\left( \frac{\alpha 2^{\alpha +1}+1}{2^{\alpha +2}\left( \alpha
+1\right) \left( \alpha +2\right) }\left\vert f^{\prime }(a)\right\vert
^{q}+m\left( \frac{1}{8}-\frac{\alpha 2^{\alpha +1}+1}{2^{\alpha +2}\left(
\alpha +1\right) \left( \alpha +2\right) }\right) \left\vert f^{\prime
}(b)\right\vert ^{q}\right) ^{\frac{1}{q}}\right\} .
\end{eqnarray*}
\end{corollary}

\begin{corollary}
Under the assumptions of Theorem \ref{2.2} with $\mu =\frac{1}{2}$ and $%
\lambda =0$, we have%
\begin{eqnarray*}
&&\left\vert f\left( \frac{a+mb}{2}\right) -\frac{1}{mb-a}%
\dint\limits_{a}^{mb}f(x)dx\right\vert \\
&\leq &\left( mb-a\right) \left( \frac{1}{8}\right) ^{1-\frac{1}{q}}\times
\left\{ \left( \frac{1}{2^{\alpha +2}\left( \alpha +2\right) }\left\vert
f^{\prime }(a)\right\vert ^{q}+m\left( \frac{1}{8}-\frac{1}{2^{\alpha
+2}\left( \alpha +2\right) }\right) \left\vert f^{\prime }(b)\right\vert
^{q}\right) ^{\frac{1}{q}}\right. \\
&&\left. +\left( \frac{2^{\alpha +2}-\alpha -3}{2^{\alpha +2}\left( \alpha
+1\right) \left( \alpha +2\right) }\left\vert f^{\prime }(a)\right\vert
^{q}+m\left( \frac{1}{8}-\frac{2^{\alpha +2}-\alpha -3}{2^{\alpha +2}\left(
\alpha +1\right) \left( \alpha +2\right) }\right) \left\vert f^{\prime
}(b)\right\vert ^{q}\right) ^{\frac{1}{q}}\right\}
\end{eqnarray*}
\end{corollary}

\begin{theorem}
\label{2.4}Let $f:I\subseteq \lbrack 0,\infty )\rightarrow \mathbb{R}$ be a
differentiable mapping on $I^{\circ }$ such that $f^{\prime }\in L[a,b]$,
where $a,b\in I^{\circ }$ with $a<b$ and $\lambda ,\mu \in \left[ 0,1\right] 
$. If $\left\vert f^{\prime }\right\vert ^{q}$ is $(\alpha ,m)-$convex on $%
[a,b]$, for $(\alpha ,m)\in \left( 0,1\right] ^{2}$, $mb>a$, $q>1,$ then the
following inequality holds:%
\begin{equation}
\left\vert \lambda \left( \mu f(a)+\left( 1-\mu \right) f(mb)\right) +\left(
1-\lambda \right) f(\mu a+m\left( 1-\mu \right) b)-\frac{1}{mb-a}%
\dint\limits_{a}^{mb}f(x)dx\right\vert  \label{2-3}
\end{equation}%
\begin{equation*}
\leq \left( mb-a\right) \left( \frac{1}{p+1}\right) ^{\frac{1}{p}}.\left\{ 
\begin{array}{cc}
\left[ \vartheta _{1}^{\frac{1}{p}}A^{\frac{1}{q}}+\vartheta _{3}^{\frac{1}{p%
}}B^{\frac{1}{q}}\right] , & \lambda \left( 1-\mu \right) \leq \mu \leq
1-\lambda \mu \\ 
\left[ \vartheta _{2}^{\frac{1}{p}}A^{\frac{1}{q}}+\vartheta _{3}^{\frac{1}{p%
}}B^{\frac{1}{q}}\right] , & \mu \leq \lambda \left( 1-\mu \right) \leq
1-\lambda \mu \\ 
\left[ \vartheta _{2}^{\frac{1}{p}}A^{\frac{1}{q}}+\vartheta _{4}^{\frac{1}{p%
}}B^{\frac{1}{q}}\right] , & \lambda \left( 1-\mu \right) \leq 1-\lambda \mu
\leq \mu%
\end{array}%
\right.
\end{equation*}%
where 
\begin{equation*}
A=\mu \times \min \left\{ \frac{\left\vert f^{\prime }\left( \mu a+m\left(
1-\mu \right) b\right) \right\vert ^{q}+\alpha m\left\vert f^{\prime }\left(
b\right) \right\vert ^{q}}{\alpha +1},\frac{\left\vert f^{\prime }\left(
mb\right) \right\vert ^{q}+\alpha m\left\vert f^{\prime }\left( \frac{\mu
a+m\left( 1-\mu \right) b}{m}\right) \right\vert ^{q}}{\alpha +1}\right\}
\end{equation*}%
\begin{equation*}
\ B=\left( 1-\mu \right) \times \min \left\{ \frac{\left\vert f^{\prime
}\left( \mu a+m\left( 1-\mu \right) b\right) \right\vert ^{q}+\alpha
m\left\vert f^{\prime }\left( \frac{a}{m}\right) \right\vert ^{q}}{\alpha +1}%
,\frac{\left\vert f^{\prime }\left( a\right) \right\vert ^{q}+\alpha
m\left\vert f^{\prime }\left( \frac{\mu a+m\left( 1-\mu \right) b}{m}\right)
\right\vert ^{q}}{\alpha +1}\right\}
\end{equation*}%
\begin{eqnarray}
\vartheta _{1} &=&\left[ \lambda \left( 1-\mu \right) \right] ^{p+1}+\left[
\mu -\lambda \left( 1-\mu \right) \right] ^{p+1},\ \vartheta _{2}=\left[
\lambda \left( 1-\mu \right) \right] ^{p+1}-\left[ \lambda \left( 1-\mu
\right) -\mu \right] ^{p+1},  \notag \\
\vartheta _{3} &=&\left[ 1-\lambda \mu -\mu \right] ^{p+1}+\left[ \lambda
\mu \right] ^{p+1},\ \vartheta _{4}=\left[ \lambda \mu \right] ^{p+1}-\left[
\mu -1+\lambda \mu \right] ^{p+1},  \notag
\end{eqnarray}%
and $\frac{1}{p}+\frac{1}{q}=1.$
\end{theorem}

\begin{proof}
From Lemma \ref{2.1} and by H\"{o}lder's integral inequality, we have%
\begin{eqnarray}
&&\left\vert \lambda \left( \mu f(a)+\left( 1-\mu \right) f(mb)\right)
+\left( 1-\lambda \right) f(\mu a+m\left( 1-\mu \right) b)-\frac{1}{mb-a}%
\dint\limits_{a}^{mb}f(x)dx\right\vert  \notag \\
&\leq &\left( mb-a\right) \left[ \dint\limits_{0}^{\mu }\left\vert
-t+\lambda \left( 1-\mu \right) \right\vert \left\vert f^{\prime }\left(
ta+m(1-t)b\right) \right\vert dt\right.  \notag \\
&&\left. +\dint\limits_{\mu }^{1}\left\vert -t+\left( 1-\alpha \lambda
\right) \right\vert \left\vert f^{\prime }\left( ta+m(1-t)b\right)
\right\vert dt\right]  \notag \\
&\leq &\left( mb-a\right) \left\{ \left( \dint\limits_{0}^{\mu }\left\vert
-t+\lambda \left( 1-\mu \right) \right\vert ^{p}dt\right) ^{\frac{1}{p}%
}\left( \dint\limits_{0}^{\mu }\left\vert f^{\prime }\left(
ta+m(1-t)b\right) \right\vert ^{q}dt\right) ^{\frac{1}{q}}\right.  \notag \\
&&\left. +\left( \dint\limits_{\mu }^{1}\left\vert -t+\left( 1-\lambda \mu
\right) \right\vert ^{p}dt\right) ^{\frac{1}{p}}\left( \dint\limits_{\mu
}^{1}\left\vert f^{\prime }\left( ta+m(1-t)b\right) \right\vert
^{q}dt\right) ^{\frac{1}{q}}\right\}  \label{2-3a}
\end{eqnarray}%
Since $\left\vert f^{\prime }\right\vert ^{q}$ is $(\alpha ,m)-$convex on $%
[a,b]$, for $(\alpha ,m)\in \left( 0,1\right] ^{2}$ and $\mu \in \left( 0,1%
\right] $ by the inequality (\ref{1-3}), we get%
\begin{equation}
\dint\limits_{0}^{\mu }\left\vert f^{\prime }\left( ta+m(1-t)b\right)
\right\vert ^{q}dt=\mu \left[ \frac{1}{\mu \left( mb-a\right) }%
\dint\limits_{\mu a+m\left( 1-\mu \right) b}^{mb}\left\vert f^{\prime
}\left( x\right) \right\vert ^{q}dx\right] \leq  \label{2-3b}
\end{equation}%
\begin{equation*}
\mu \times \min \left\{ \frac{\left\vert f^{\prime }\left( \mu a+m\left(
1-\mu \right) b\right) \right\vert ^{q}+\alpha m\left\vert f^{\prime }\left(
b\right) \right\vert ^{q}}{\alpha +1},\frac{\left\vert f^{\prime }\left(
mb\right) \right\vert ^{q}+\alpha m\left\vert f^{\prime }\left( \frac{\mu
a+m\left( 1-\mu \right) b}{m}\right) \right\vert ^{q}}{\alpha +1}\right\} .
\end{equation*}%
The inequality (\ref{2-3b}) holds for $\mu =0$ too. Similarly, for $\mu \in %
\left[ 0,1\right) $ by the inequality (\ref{1-3}), we have%
\begin{equation}
\dint\limits_{\mu }^{1}\left\vert f^{\prime }\left( ta+m(1-t)b\right)
\right\vert ^{q}dt=\left( 1-\mu \right) \left[ \frac{1}{\left( 1-\mu \right)
\left( mb-a\right) }\dint\limits_{a}^{\mu a+m\left( 1-\mu \right)
b}\left\vert f^{\prime }\left( x\right) \right\vert ^{q}dx\right] \leq
\label{2-3c}
\end{equation}%
\begin{equation*}
\left( 1-\mu \right) \times \min \left\{ \frac{\left\vert f^{\prime }\left(
\mu a+m\left( 1-\mu \right) b\right) \right\vert ^{q}+\alpha m\left\vert
f^{\prime }\left( \frac{a}{m}\right) \right\vert ^{q}}{\alpha +1},\frac{%
\left\vert f^{\prime }\left( a\right) \right\vert ^{q}+\alpha m\left\vert
f^{\prime }\left( \frac{\mu a+m\left( 1-\mu \right) b}{m}\right) \right\vert
^{q}}{\alpha +1}\right\} .
\end{equation*}%
The inequality (\ref{2-3b}) holds for $\mu =1$ too. By simple computation%
\begin{equation}
\dint\limits_{0}^{\mu }\left\vert -t+\lambda \left( 1-\mu \right)
\right\vert ^{p}dt=\left\{ 
\begin{array}{cc}
\frac{\left[ \lambda \left( 1-\mu \right) \right] ^{p+1}+\left[ \mu -\lambda
\left( 1-\mu \right) \right] ^{p+1}}{p+1}, & \lambda \left( 1-\mu \right)
\leq \mu \\ 
\frac{\left[ \lambda \left( 1-\mu \right) \right] ^{p+1}-\left[ \lambda
\left( 1-\mu \right) -\mu \right] ^{p+1}}{p+1}, & \lambda \left( 1-\mu
\right) \geq \mu%
\end{array}%
\right. ,  \label{2-3d}
\end{equation}%
and%
\begin{equation}
\dint\limits_{\mu }^{1}\left\vert -t+\left( 1-\lambda \mu \right)
\right\vert ^{p}dt=\left\{ 
\begin{array}{cc}
\frac{\left[ 1-\lambda \mu -\mu \right] ^{p+1}+\left[ \lambda \mu \right]
^{p+1}}{p+1}, & \mu \leq 1-\lambda \mu \\ 
\frac{\left[ \lambda \mu \right] ^{p+1}-\left[ \mu -1+\lambda \mu \right]
^{p+1}}{p+1}, & \mu \geq 1-\lambda \mu%
\end{array}%
\right. ,  \label{2-3e}
\end{equation}%
thus, using (\ref{2-3b})-(\ref{2-3e}) in (\ref{2-3a}), we obtain the
inequality (\ref{2-3}). This completes the proof.
\end{proof}

\begin{corollary}
\label{2.5}Under the assumptions of Theorem \ref{2.4} with $\mu =\frac{1}{2}$
and $\lambda =\frac{1}{3}$, we have%
\begin{eqnarray*}
&&\left\vert \frac{1}{6}\left[ f(a)+4f\left( \frac{a+mb}{2}\right) +f(mb)%
\right] -\frac{1}{mb-a}\dint\limits_{a}^{mb}f(x)dx\right\vert \\
&\leq &\left( \frac{mb-a}{12}\right) \left( \frac{2^{p+1}+1}{3\left(
p+1\right) }\right) ^{\frac{1}{p}}\left( A_{1}^{\frac{1}{q}}+B_{1}^{\frac{1}{%
q}}\right) ,
\end{eqnarray*}%
where%
\begin{equation*}
A_{1}=\min \left\{ \frac{\left\vert f^{\prime }\left( \frac{a+mb}{2}\right)
\right\vert ^{q}+\alpha m\left\vert f^{\prime }\left( b\right) \right\vert
^{q}}{\alpha +1},\frac{\left\vert f^{\prime }\left( mb\right) \right\vert
^{q}+\alpha m\left\vert f^{\prime }\left( \frac{a+mb}{2m}\right) \right\vert
^{q}}{\alpha +1}\right\} ,
\end{equation*}%
and%
\begin{equation*}
B_{1}=\min \left\{ \frac{\left\vert f^{\prime }\left( \frac{a+mb}{2}\right)
\right\vert ^{q}+\alpha m\left\vert f^{\prime }\left( \frac{a}{m}\right)
\right\vert ^{q}}{\alpha +1},\frac{\left\vert f^{\prime }\left( a\right)
\right\vert ^{q}+\alpha m\left\vert f^{\prime }\left( \frac{a+mb}{2m}\right)
\right\vert ^{q}}{\alpha +1}\right\} .
\end{equation*}
\end{corollary}

\begin{remark}
In Corollary \ref{2.5}, if we take $\alpha =m=1,$ then we obtain the
following inequality%
\begin{eqnarray*}
&&\left\vert \frac{1}{6}\left[ f(a)+4f\left( \frac{a+b}{2}\right) +f(b)%
\right] -\frac{1}{b-a}\dint\limits_{a}^{b}f(x)dx\right\vert \leq \left( 
\frac{b-a}{12}\right) \left( \frac{1+2^{p+1}}{3\left( p+1\right) }\right) ^{%
\frac{1}{p}} \\
&&\times \left\{ \left( \frac{\left\vert f^{\prime }\left( \frac{a+b}{2}%
\right) \right\vert ^{q}+\left\vert f^{\prime }\left( a\right) \right\vert
^{q}}{2}\right) ^{\frac{1}{q}}+\left( \frac{\left\vert f^{\prime }\left( 
\frac{a+b}{2}\right) \right\vert ^{q}+\left\vert f^{\prime }\left( b\right)
\right\vert ^{q}}{2}\right) ^{\frac{1}{q}}\right\} ,
\end{eqnarray*}%
which is the same of the inequality in \cite[Corollary 3]{SSO10}
\end{remark}

\begin{corollary}
Under the assumptions of Theorem \ref{2.4} with $\mu =\frac{1}{2}$ and $%
\lambda =1$, we have%
\begin{eqnarray*}
&&\left\vert \frac{f(a)+f(mb)}{2}-\frac{1}{mb-a}\dint\limits_{a}^{mb}f(x)dx%
\right\vert \\
&\leq &\left( \frac{mb-a}{4}\right) \left( \frac{1}{p+1}\right) ^{\frac{1}{p}%
}\left( A_{1}^{\frac{1}{q}}+B_{1}^{\frac{1}{q}}\right) .
\end{eqnarray*}
\end{corollary}

\begin{corollary}
Under the assumptions of Theorem \ref{2.4} with $\mu =\frac{1}{2}$ and $%
\lambda =0$, we have%
\begin{eqnarray*}
&&\left\vert f\left( \frac{a+mb}{2}\right) -\frac{1}{mb-a}%
\dint\limits_{a}^{mb}f(x)dx\right\vert \\
&\leq &\left( \frac{mb-a}{4}\right) \left( \frac{1}{p+1}\right) ^{\frac{1}{p}%
}\left( A_{1}^{\frac{1}{q}}+B_{1}^{\frac{1}{q}}\right) .
\end{eqnarray*}
\end{corollary}

\begin{theorem}
\label{2.6}Let $f:I\subseteq \lbrack 0,\infty )\rightarrow \mathbb{R}$ be a
differentiable mapping on $I^{\circ }$ such that $f^{\prime }\in L[a,b]$,
where $a,b\in I^{\circ }$ with $a<b$ and $\lambda ,\mu \in \left[ 0,1\right] 
$. If $\left\vert f^{\prime }\right\vert ^{q}$ is $(\alpha ,m)-$convex on $%
[a,b]$, for $(\alpha ,m)\in \left( 0,1\right] ^{2}$, $mb>a$, $q>1,$ then the
following inequality holds:%
\begin{eqnarray}
&&\left\vert \lambda \left( \mu f(a)+\left( 1-\mu \right) f(mb)\right)
+\left( 1-\lambda \right) f(\mu a+m\left( 1-\mu \right) b)-\frac{1}{mb-a}%
\dint\limits_{a}^{mb}f(x)dx\right\vert  \label{2-4} \\
&\leq &\left( mb-a\right) \left( \frac{1}{p+1}\right) ^{\frac{1}{p}}.\left\{ 
\begin{array}{cc}
\left[ \vartheta _{1}^{\frac{1}{p}}A_{2}^{\frac{1}{q}}+\vartheta _{3}^{\frac{%
1}{p}}B_{2}^{\frac{1}{q}}\right] , & \lambda \left( 1-\mu \right) \leq \mu
\leq 1-\lambda \mu \\ 
\left[ \vartheta _{2}^{\frac{1}{p}}A_{2}^{\frac{1}{q}}+\vartheta _{3}^{\frac{%
1}{p}}B_{2}^{\frac{1}{q}}\right] , & \mu \leq \lambda \left( 1-\mu \right)
\leq 1-\lambda \mu \\ 
\left[ \vartheta _{2}^{\frac{1}{p}}A_{2}^{\frac{1}{q}}+\vartheta _{4}^{\frac{%
1}{p}}B_{2}^{\frac{1}{q}}\right] , & \lambda \left( 1-\mu \right) \leq
1-\lambda \mu \leq \mu%
\end{array}%
\right. ,  \notag
\end{eqnarray}%
where%
\begin{equation*}
A_{2}=\frac{\mu ^{\alpha +1}\left\vert f^{\prime }\left( a\right)
\right\vert ^{q}+m\left[ \mu \left( \alpha +1\right) -\mu ^{\alpha +1}\right]
\left\vert f^{\prime }\left( b\right) \right\vert ^{q}}{\alpha +1}
\end{equation*}%
and%
\begin{equation*}
B_{2}=\frac{\left( 1-\mu ^{\alpha +1}\right) \left\vert f^{\prime }\left(
a\right) \right\vert ^{q}+m\left[ \left( \mu ^{\alpha +1}-1\right) +\left(
1-\mu \right) \left( \alpha +1\right) \right] \left\vert f^{\prime }\left(
b\right) \right\vert ^{q}}{\alpha +1}.
\end{equation*}
\end{theorem}

\begin{proof}
From Lemma \ref{2.1} and by H\"{o}lder's integral inequality, we have the
inequality (\ref{2-3a}). Since Since $\left\vert f^{\prime }\right\vert ^{q}$
is $(\alpha ,m)-$convex on $[a,b],$ we know that for $t\in \left[ 0,1-\alpha %
\right] $ and $t\in \left[ 1-\alpha ,1\right] $ 
\begin{equation*}
\left\vert f^{\prime }\left( ta+m(1-t)b\right) \right\vert ^{q}\leq
t^{\alpha }\left\vert f^{\prime }\left( a\right) \right\vert
^{q}+m(1-t^{\alpha })\left\vert f^{\prime }\left( b\right) \right\vert ^{q}.
\end{equation*}%
Hence 
\begin{eqnarray*}
&&\left\vert \lambda \left( \mu f(a)+\left( 1-\mu \right) f(mb)\right)
+\left( 1-\lambda \right) f(\mu a+m\left( 1-\mu \right) b)-\frac{1}{mb-a}%
\dint\limits_{a}^{mb}f(x)dx\right\vert \\
&\leq &\left( mb-a\right) \left\{ \left( \dint\limits_{0}^{\mu }\left\vert
-t+\lambda \left( 1-\mu \right) \right\vert ^{p}dt\right) ^{\frac{1}{p}%
}\left( \dint\limits_{0}^{\mu }\left[ t^{\alpha }\left\vert f^{\prime
}\left( a\right) \right\vert ^{q}+m(1-t^{\alpha })\left\vert f^{\prime
}\left( b\right) \right\vert ^{q}\right] dt\right) ^{\frac{1}{q}}\right. \\
&&\left. +\left( \dint\limits_{\mu }^{1}\left\vert -t+\left( 1-\alpha
\lambda \right) \right\vert ^{p}dt\right) ^{\frac{1}{p}}\left(
\dint\limits_{\mu }^{1}\left[ t^{\alpha }\left\vert f^{\prime }\left(
a\right) \right\vert ^{q}+m(1-t^{\alpha })\left\vert f^{\prime }\left(
b\right) \right\vert ^{q}\right] dt\right) ^{\frac{1}{q}}\right\} \\
&\leq &\left( mb-a\right) \left\{ \left( \dint\limits_{0}^{\mu }\left\vert
-t+\lambda \left( 1-\mu \right) \right\vert ^{p}dt\right) ^{\frac{1}{p}%
}\left( \frac{\mu ^{\alpha +1}\left\vert f^{\prime }\left( a\right)
\right\vert ^{q}+m\left[ \mu \left( \alpha +1\right) -\mu ^{\alpha +1}\right]
\left\vert f^{\prime }\left( b\right) \right\vert ^{q}}{\alpha +1}\right) ^{%
\frac{1}{q}}\right.
\end{eqnarray*}%
\begin{equation}
\left. +\left( \dint\limits_{\mu }^{1}\left\vert -t+\left( 1-\alpha \lambda
\right) \right\vert ^{p}dt\right) ^{\frac{1}{p}}\left( \frac{\left( 1-\mu
^{\alpha +1}\right) \left\vert f^{\prime }\left( a\right) \right\vert ^{q}+m%
\left[ \left( \mu ^{\alpha +1}-1\right) +\left( 1-\mu \right) \left( \alpha
+1\right) \right] \left\vert f^{\prime }\left( b\right) \right\vert ^{q}}{%
\alpha +1}\right) ^{\frac{1}{q}}\right\} .  \label{2-4a}
\end{equation}%
thus, using (\ref{2-3d}),(\ref{2-3e}) in (\ref{2-4a}), we obtain the
inequality (\ref{2-4}). This completes the proof.
\end{proof}

\begin{corollary}
Let the assumptions of Theorem \ref{2.6} hold. Then for $\mu =\frac{1}{2}$
and $\lambda =\frac{1}{3}$, from the inequality (\ref{2-4}) we get the
following Simpson type inequality%
\begin{eqnarray*}
&&\left\vert \frac{1}{6}\left[ f(a)+4f\left( \frac{a+mb}{2}\right) +f(mb)%
\right] -\frac{1}{mb-a}\dint\limits_{a}^{mb}f(x)dx\right\vert \\
&\leq &\left( \frac{mb-a}{12}\right) \left( \frac{2^{p+1}+1}{3\left(
p+1\right) }\right) ^{\frac{1}{p}}\left( \frac{1}{2^{\alpha }\left( \alpha
+1\right) }\right) ^{\frac{1}{q}}\left( A_{3}^{\frac{1}{q}}+B_{3}^{\frac{1}{q%
}}\right) .
\end{eqnarray*}%
where%
\begin{equation*}
A_{3}=\left\vert f^{\prime }\left( a\right) \right\vert ^{q}+m\left[
2^{\alpha }\left( \alpha +1\right) -1\right] \left\vert f^{\prime }\left(
b\right) \right\vert ^{q}
\end{equation*}%
and%
\begin{equation*}
B_{3}=\left( 2^{\alpha +1}-1\right) \left\vert f^{\prime }\left( a\right)
\right\vert ^{q}+m\left[ 2^{\alpha }\left( \alpha +1\right) +1-2^{\alpha +1}%
\right] \left\vert f^{\prime }\left( b\right) \right\vert ^{q}.
\end{equation*}
\end{corollary}

\begin{corollary}
\label{2.7}Let the assumptions of Theorem \ref{2.6} hold. Then for $\mu =%
\frac{1}{2}$ and $\lambda =1$, from the inequality (\ref{2-4}) we get the
following Simpson type inequality%
\begin{eqnarray*}
&&\left\vert \frac{f(a)+f(mb)}{2}-\frac{1}{mb-a}\dint\limits_{a}^{mb}f(x)dx%
\right\vert \\
&\leq &\left( \frac{mb-a}{4}\right) \left( \frac{1}{\left( p+1\right) }%
\right) ^{\frac{1}{p}}\left( \frac{1}{2^{\alpha }\left( \alpha +1\right) }%
\right) ^{\frac{1}{q}}\left( A_{3}^{\frac{1}{q}}+B_{3}^{\frac{1}{q}}\right) .
\end{eqnarray*}
\end{corollary}

\begin{corollary}
In Corollary \ref{2.7}, if we take $\alpha =1$ we obtain the following
inequality%
\begin{eqnarray}
&&\left\vert \frac{f(a)+f(mb)}{2}-\frac{1}{mb-a}\dint\limits_{a}^{mb}f(x)dx%
\right\vert \leq \left( mb-a\right) \left( \frac{1}{p+1}\right) ^{\frac{1}{p}%
}  \notag \\
&&\times \left( \frac{1}{4}\right) ^{1+\frac{1}{q}}\left[ \left( \left\vert
f^{\prime }\left( a\right) \right\vert ^{q}+3m\left\vert f^{\prime }\left(
b\right) \right\vert ^{q}\right) ^{\frac{1}{q}}+\left( 3\left\vert f^{\prime
}\left( a\right) \right\vert ^{q}+m\left\vert f^{\prime }\left( b\right)
\right\vert ^{q}\right) ^{\frac{1}{q}}\right]  \notag
\end{eqnarray}%
\begin{equation}
\leq \left( mb-a\right) \left( \frac{1}{4}\right) ^{1+\frac{1}{q}}\left[
\left( \left\vert f^{\prime }\left( a\right) \right\vert ^{q}+3m\left\vert
f^{\prime }\left( b\right) \right\vert ^{q}\right) ^{\frac{1}{q}}+\left(
3\left\vert f^{\prime }\left( a\right) \right\vert ^{q}+m\left\vert
f^{\prime }\left( b\right) \right\vert ^{q}\right) ^{\frac{1}{q}}\right]
\label{2-5}
\end{equation}%
where we have used the fact that $1/2<\left( 1/\left( p+1\right) \right)
^{1/p}<1.$ We note that the inequality (\ref{2-5}) is the same of the
inequality in \cite[Corollary 2.7 (i)]{P12}.
\end{corollary}

\begin{corollary}
Under the assumptions of Theorem \ref{2.6} with $\mu =\frac{1}{2}$ and $%
\lambda =0$, we have%
\begin{eqnarray*}
&&\left\vert f\left( \frac{a+mb}{2}\right) -\frac{1}{mb-a}%
\dint\limits_{a}^{mb}f(x)dx\right\vert \\
&\leq &\left( \frac{mb-a}{4}\right) \left( \frac{1}{\left( p+1\right) }%
\right) ^{\frac{1}{p}}\left( \frac{1}{2^{\alpha }\left( \alpha +1\right) }%
\right) ^{\frac{1}{q}}\left( A_{3}^{\frac{1}{q}}+B_{3}^{\frac{1}{q}}\right) .
\end{eqnarray*}
\end{corollary}


\begin{thebibliography}{99}
\bibitem{BOP08} Bakula, M.K., Ozdemir, M.E. and Pecaric,J.: \textit{Hadamard
type inequalities for }$m$\textit{-convex and }$(\alpha ,m)$\textit{-convex
functions}, J. Inequal. Pure Appl. Math. 9, no.4, Article 96, p. 12, 2008.
[Online: http://jipam.vu.edu.au].

\bibitem{M93} Mihe\c{s}an, V.G.: \textit{A generalization of the convexity, }%
Seminer on Functional Equations, Approximation and Convexity, Cluj-Napoca,
Romania (1993).

\bibitem{OAK11} Ozdemir, M.E., Avc\i , M. and Kavurmac\i , H.: \textit{%
Hermite-Hadamard-type inequalities via }$(\alpha ,m)$\textit{-convexity, }%
Computers and Mathematics with Applications 61, 2614-2620, 2011.

\bibitem{OKS10} Ozdemir, M.E., Kavurmac\i , H. and Set, E.:\ \textit{%
Ostrowski's type inequalities for }$(\alpha ,m)$\textit{-convex functions},
Kyungpook Math. J. 50, 371-378, 2010.

\bibitem{OSS11} Ozdemir, M.E., Set, E .and Sar\i kaya, M.Z.: \textit{Some
new Hadamard's type inequalities for co-ordinated }$m$\textit{-convex and }$%
(\alpha ,m)$\textit{-convex functions,} Hacettepe Journal of Mathematics and
Statistics volume 40 (2), 219-229, 2011.

\bibitem{P12} J. Park,\textit{\ Hermite-Hadamard and Simpson-like type
inequalities for differentiable }$(\alpha ,m)$\textit{-convex mappings,}
Int. Journal of Math. and Mathematical Sciences, vol. 2012, Article ID
809689, 12 pages.

\bibitem{SA11} Sarikaya, M.Z., Aktan, N.: \textit{On the generalization of
some integral inequalities and their applications}, Mathematical and
Computer Modelling, 54 (2011) 2175-2182.

\bibitem{SSO10} Sarikaya, M.Z., Set, E., \"{O}zdemir, M.E.: \textit{On new
inequalities of Simpson's type for }$s-$\textit{convex functions}, Computers
and Mathematics with Applications 60 (2010) 2191-2199.

\bibitem{SSO10a} Sarikaya, M.Z., Set, E., \"{O}zdemir, M.E.: \textit{On new
inequalities of Simpson's type for convex functions,} RGMIA Res. Rep. Coll.
13 (2) (2010) Article 2.

\bibitem{SSOR09} Set,E., Sardari, M., Ozdemir, M.E. and Rooin, J.: \textit{%
On generalizations of the Hadamard inequality for }$(\alpha ,m)$\textit{%
-convex functions,} RGMIA Res. Rep. Coll. 12(4), Article 4, 2009.
\end{thebibliography}
\end{document}